\documentclass{amsart}
\usepackage{latexsym,amssymb,amsmath,amsthm,latexsym}
\usepackage{amsfonts}
\theoremstyle{definition}

\newtheorem{example}{Example}
\newtheorem{corollary}{Corollary}
\newtheorem{note}{Note}
\newtheorem{proposition}{Proposition}
\newtheorem{lemma}{Lemma}
\newtheorem{remark}{Remark}

\newtheorem{identity}{Identity}
\title{Pascal Triangle and Restricted Words}
\author{Milan Janji\'c}
\date{\today}
\begin{document}
\maketitle
\begin{center}Department for Mathematics and Informatics, University of Banja
Luka,\end{center}
 \begin{center}Republic of Srpska, BA\end{center}
\begin{abstract} We continue  to investigate combinatorial properties of functions
$f_m$ and $c_m$ considered in our previous papers. They depend on an initial arithmetic function $f_0$.
In this paper, the values of $f_0$ are the binomial coefficients.

 We first consider the case when the values of $f_0$ are the binomial coefficients from a row of the Pascal triangle. The values of $f_0$ consider next  are the binomial coefficients from a diagonal of the Pascal triangle. In two final cases, the values of $f_0$ are the central binomial coefficients and its adjacent neighbors. In each case, we derive an explicit formula for $c_1(n,k)$ and give its interpretation in terms of restricted words. In the first two cases, we also consider the functions $f_m$ and $c_m$, for $(m>0)$.
  \end{abstract}

\section{Introduction} In this paper, we continue to investigate the problem of the enumeration  of restricted words. Previously, the functions $f_m(n)$ and $c_m(n,k)$ were defined as follows. For an initial arithmetic function $f_0$, $f_m(m>0)$ is the $m$th invert transform of $f_0$. The function $c_m(n,k)$ was defined in the following  way:
\begin{equation}\label{cmnk}c_m(n,k)=\sum_{i_1+i_2+\cdots+i_k=n}f_{m-1}(i_1)\cdot f_{m-1}(i_2)\cdots f_{m-1}(i_k),\end{equation}
 where the sum is over positive $i_1,i_2,\ldots,i_k$.
  Its connections with the problem of the enumeration of restricted words were consider.
A number of results of this kind is obtained in Janji\'c~\cite{ja1,ja2,ja3}. Some results  have also been considered by other authors, for instance, in~\cite{bir1,bir2,jma,mw}.

 In this paper, we derive results for four types of initial functions,  which have different kind  of the binomial coefficients as values.  In the first case, the values of $f_0$ are the binomial coefficients from a row of the Pascal triangle. In the second case, the values  are  from a diagonal of the Pascal triangle. In the remaining cases, the values of $f_0$ are the central binomial coefficients and its adjacent neighbors.

In the first two cases, we describe the restricted words counted by $f_m$ and $c_m$, and derive explicit formulas for these functions.
In the last two cases, we consider the function $c_1$ only.

\section{Rows of the Pascal triangle}
For a positive integer $a$, we want to enumerate words over the finite alphabet $\{0,1,\ldots,a-1,\ldots\}$, in which letters in words over $\{0,1,\ldots,a-1\}$ are arranged into ascending order. We let $\mathcal P_1$ denote this property.

We define
 \begin{equation*}f_0(n)={a\choose n-1},(n=1,2,\ldots).\end{equation*}
 It is clear that $f_0(n)$ is the number of words of length  $n-1$ over the alphabet $\{0,1,\ldots,a-1\}$ satisfying $\mathcal P_1$. Since $f_0(1)=1$,  using~Janji\'c~\cite[Proposition 12]{ja3}, we obtain the following combinatorial meaning of $c_m(n,k)$.

\begin{corollary}The number of words of length $n-1$ over the alphabet $\{0,1,\ldots,a-1,\ldots, a+m-1\}$ having $k-1$ letters equal to $a+m-1$ and satisfying $\mathcal P_1$ is $c_m(n,k)$.
\end{corollary}
According to~\cite[Corollary 2]{ja3}, we get
\begin{corollary}The number of words of length $n-1$ over the alphabet    $\{0,1,\ldots,a-1,\ldots, a+m-1\}$ satisfying $\mathcal P_1$ is $f_m(n)$.
\end{corollary}
 We next derive an explicit formula for $c_1(n,k)$.
\begin{proposition}\label{pb1} We have
 \begin{equation*}c_1(n,k)={ak\choose n-k}.\end{equation*}
  \end{proposition}
 \begin{proof} We use the induction on $k$. From Janji\'c~\cite[Equation (3)]{ja2}, we have
 $c_1(n,1)=f_0(n)$, which means that the statement holds for $k=1$. Suppose that the statement is true for $k-1,(k>1)$. Using the recurrence~\cite[Equation (3)]{ja3}, and the induction hypothesis, we obtain
\[ c_1(n,k)=\sum_{i=1}^{n-k+1}{a\choose i-1}{ak-a\choose n-i-k+1},\]
and the statement follows from the Vandermonde convolution.
 \end{proof}

As a consequence of~\cite[Equation (9), Proposition 7]{ja3}, we obtain the following explicit formulas for  $c_m(n,k)$ and $f_m$.
\begin{corollary}
We have
\begin{equation*}c_m(n,k)=\sum_{i=k}^n(m-1)^{i-k}{i-1\choose k-1}{ia\choose n-i},\end{equation*}
\begin{equation*}f_m(n)=\sum_{i=1}^nm^{i-1}{ia\choose n-i}.\end{equation*}
\end{corollary}
\begin{corollary} In the case $a=2$, the sequence $f_1(1),f_1(2),\ldots$ is seqnum{A002478}, which is the bisection of the Narayana's cows sequence (seqnum{A000930}).
\end{corollary}

\section{Diagonals of Pascal triangle}
The following problem which we investigate is:
For a positive integer $a$, find the numbers of words over the alphabet $\{0,1,\ldots,a-1,\ldots\}$ such that subwords from $\{0,1,\ldots,a-1\}$ have no rises. We let $\mathcal P_2$ denote this property.
We show that for this problem, the values of the initial function are  figurate numbers, that is, the numbers on a diagonal of the Pascal triangle.
We let $d(n-1,a)$ denote the number words length $n-1$ satisfying $\mathcal P_2$.
\begin{proposition} The following equation holds:
\begin{equation}\label{fgb}d(n-1,a)={n+a-2\choose a-1}.\end{equation}
\end{proposition}
\begin{proof} We first have $d(0,a)=1$, since the empty word has no a rise. Assume that $n>1$. The following recurrence  holds:
\begin{equation*}d(n,a+1)=d(n,a)+d(n-1,a+1),(n>0).\end{equation*}
Namely, each word of length $n-1$ over $\{0,1,\ldots,a-1,a\}$ that has no rises may begin with a letter from $\{0,1,\ldots,a-1\}$. The letter $a$ does not appear in a word, which yields that there are $d(n-1,a)$ such words.
There remains the words  of length $n-1$ beginning with $a$. Obviously, there are
$d(n-2,a+1)$ such words.
It follows that
\begin{equation*}\begin{array}{ccc}d(n-1,a+1)-d(n-2,a+1)&=&d(n-1,a),\\
d(n-2,a+1)-d(n-3,a+1)&=&d(n-2,a),\\
&\vdots&\\
d(1,a+1)-d(0,a+1)&=&d(1,a).
\end{array}\end{equation*}
Adding expressions on the left-hand sides and the right-hand sides, we obtain the following recurrence:
\begin{equation}\label{ww}d(n-1,a+1)=\sum_{i=1}^{n}d(i-1,a).\end{equation}
To prove (\ref{fgb}), we use induction on $a$. If $a=1$, then $d(n-1,1)=1$, since the alphabet consists of the empty word.
 Assume that the statement holds for $a\geq 1$. Then (\ref{ww}) takes the form
\begin{equation*}
d(n-1,a+1)=\sum_{i=1}^n{i+a-2\choose a-1},
\end{equation*}
and the statement holds according to the  horizontal recurrence for the binomial coefficients.
\end{proof}
Therefore, in the case $f_0(n)={n+a-2\choose a-1},(n=1,2,\ldots)$,
since $f_0(1)=1$, using Janji\'c~\cite[Proposition 12]{ja3}, we obtain
\begin{corollary}The number $c_m(n,k),(m\geq 1)$ is the number of words of length $n-1$ over the alphabet $\{0,1,\ldots,a-1,\ldots, a+m-1\}$ having $k-1$ letters equal to $a+m-1$ and
    satisfying $\mathcal P_2$.
\end{corollary}
According to~\cite[Corollary 2]{ja3}, we get
\begin{corollary}The number $f_m(n)$ equals the number of words of length $n-1$ over the alphabet $\{0,1,\ldots,a-1,\ldots, a+m-1\}$ satisfying $\mathcal P_2$.
\end{corollary}
To derive an explicit formula for $c_1(n,k)$, we need the following lemma.
\begin{lemma} Let $u\geq v\geq w\geq 1$ be integers. Then
\begin{equation}\label{idd1}{u\choose
v}=\sum_{i=w}^{u-v+w}{i-1\choose w-1}{u-i\choose v-w}.\end{equation}
\end{lemma}
\begin{proof} We know that ${u\choose v}$ is the number of binary words of length $u$ with $v$ zeros. We let $i$ denote the position of the $w$th zero in a word.
It is clear that $w\leq i\leq u-v+w$. For a fixed $i$, the number of words is ${i-1\choose w-1}{u-i\choose v-w}$. Summing over all $i$, we obtain (\ref{idd1}).
\end{proof}
\begin{note}The identity (\ref{idd1}) generalizes the horizontal recurrence for the binomial coefficients, which we obtain for either $w=1$ or $w=v$. \end{note}

Next, we derive a formula for $c_1(n,k)$.
\begin{proposition}\label{cc} We have
\[c_1(n,k)={n+ak-k-1\choose ak-1}.\]
\end{proposition}
\begin{proof}
We use induction on $k$. From~\cite[Equation (3)]{ja3}, we have
\begin{equation*}c_1(n,1)=f_0(n)={n+a-2\choose a-1},\end{equation*} which means that the statement is true for $k=1$
  Using induction, we  conclude that the  statement  is equivalent to the following
  binomial identity:
\[{n+ak-k-1\choose ak-1}=\sum_{i=1}^{n-k+1}{i+a-2\choose a-1}{n+ak-k-a-i\choose ak-a-1}.\]
We prove that this identity follows from Identity \ref{idd1}. Namely, taking $w+1$ instead
of $w$ in Identity \ref{idd1}, and then replacing $i-1$ by $j$ yields
\[{u\choose v}=\sum_{j=w}^{u-v+w}{j\choose w}{u-j-1\choose v-w-1}.\]
Taking, in particular,  $w=a-1$ and replacing $j$ by $i+a-2$ implies
\[{u\choose v}=\sum_{i=1}^{u-v+1}{i+a-2\choose a-1}{u-1-i\choose v-a}.\]
Finally, taking $u=n+ak-k-1,v=ak-1$, we obtain the desired result.
\end{proof}
\begin{corollary}
\begin{enumerate} The following formulas  holds:
\item \begin{equation*}c_m(n,k)=\sum_{i=k}^n(m-1)^{i-k}{i-1\choose k-1}{n+ai-i-1\choose ai-1}.\end{equation*}
\item \begin{equation*}f_m(n)=\sum_{k=1}^nm^{i-1}{n+ai-i-1\choose ai-1}.\end{equation*}
\end{enumerate}
\end{corollary}
\begin{proof}
The proof follows from~\cite[Equation (9), Proposition 7]{ja3}.
\end{proof}
We state some particular cases. Note that the case $a=2$ was considered  in~\cite[Example 31]{ja3}.
\begin{example} In the case $a=3$, we have $f_0(n)={n+1\choose 2},(n\geq 1)$.
Hence, $f_0(n)$ is the $n$th triangular number. We thus obtain
\begin{enumerate}
\item
 The number ${n+2k-1\choose 3k-1}$ is the number of ternary words of length $n-1$  having $k-1$ letters equal to $2$, and avoiding $01,02,12$.
\item
 The number \begin{equation*}\sum_{i=k}^n(m-1)^{i-k}{i-1\choose k-1}{n+2i-1\choose 3i-1}\end{equation*}
 is the number of words of length $n-1$ over the alphabet $\{0,1,\ldots,m+2\}$ having $k-1$ letters equal to $m+2$ and avoiding $01,02,12$.
\item
The number
\begin{equation*}\sum_{i=1}^nm^{i-1}{n+2i-1\choose 3i-1}\end{equation*}
 is the number of words of length $n-1$ over the alphabet $\{0,1,\ldots,m+2\}$ avoiding $01,02,12$.
\end{enumerate}
\end{example}

This case is also related with the enumeration of some Dyck paths.
Namely, the sequence $f_0(1),f_0(2),\ldots$ makes
the second column of the Narayana triangle. It follows that $f_0(n)$ is the
number of Dyck paths of semilength $n+1$ having exactly two peaks.
In the  formula
\begin{equation*}c_1(n,k)=\sum_{i_1+i_2+\cdots+i_k=n}f_0(i_1)\cdot f_0(i_2)\cdots f_0(i_k),\end{equation*}
the product $f_0(i_1)\cdot f_0(i_2)\cdots f_0(i_k)$ equals the number of Dyck paths
of semilength $n+k$ obtained by the concatenation of $k$ Dyck paths with exactly two peaks.
We thus obtain the following Euler-type identity.
\begin{identity}
The following sets have the same number of elements
\begin{enumerate}
\item
The set of quaternary words of length $n-1$  in
which $k-1$ letters equal $3$ and which avoid $01,02,12$.
\item The set  Dyck paths of
semilength $n+k$ obtained by the concatenation of $k$ Dyck paths with exactly two peaks.
\end{enumerate}
\end{identity}
\section{Central binomial coefficients ant its adjacent neighbors }
In two concluding examples, we give combinatorial properties  for $c_1(n,k)$ only.

We start with a slight generalization of~\cite[Proposition 12]{ja3}.
Let $W$ be a set of words with a property $\mathcal P$, over a finite alphabet $\alpha$.
Assume that the empty word has the property $\mathcal P$.
For a positive integer $i$,  we denote by $W_{i-1}$ the set of words of length $l(i-1)$.
      In particular, $W_0=\emptyset$ yields $l(0)=0$. We let $f_0(i)$ denote the number of words from $W_{i-1}$. In particular, we have $f_0(1)=1$.
   Consider the equation $i_1+i_2+\cdots+i_k=n, (i_t>0,t=1,2,\ldots,k)$. For $x\not\in \alpha$, we want to count words from the alphabet $\alpha\cup\{x\}$ of the form:
   \begin{equation*}w_{i_1-1},x,w_{i_2-1},x,\ldots w_{i_{k-1}-1},x, w_{i_k-1},\end{equation*}
where $w_{i_t-1}\in W_{i_t-1},(t=1,2,\ldots,k)$. We let $N$ denote the number of such words. For fixed $i_1,i_2,\ldots,i_k$,
the word has length \begin{equation*}l(i_1-1)+\cdots+l(i_k-1)+k-1,\end{equation*} and its $k-1$ letters equal to $x$. Choosing suitable $i_1,i_2,\ldots,i_k$, each of $k-1$  letters $x$ may be put at any prescribed place in a word. Summing over all $i_1,i_2,\ldots,i_k$, implies
\begin{proposition}\label{alf} We have
\begin{equation*}N=\sum_{i_1+i_2+\cdots+i_k=n}f_0(i_1)\cdots f_0(i_k).
\end{equation*}
\end{proposition}
\begin{remark}
Taking, in particular,  $l(i-1)=i-1$ for all $i$, we obtain~\cite[Proposition 12]{ja3}.
\end{remark}
In the following two cases, we restricted our investigation to find explicit formulas for $c_1(n,k)$ and its combinatorial interpretations.
We first assume  that the values of $f_0$ are the central
binomial coefficients, that is
\begin{equation*}f_0(n)={2n-2\choose n-1}.\end{equation*}
It is clear that $f_0(n)$ is the number of binary words of length $2n-2$, in which the number of zeros equals the number of ones. We let $\mathcal P_3$ denote this condition.
In this case, we take $l(n-1)=2n-2$.
Since the empty word satisfies $\mathcal P_3$,  we have $l(0)=0$. Hence, Proposition \ref{alf} may be applied to obtain
\begin{corollary} The number $c_1(n,k)$ is the number of ternary words of length $2n-k-1$ having $k-1$ letters equal $2$, and all binary subwords satisfy $\mathcal P_3$.\end{corollary}
We derive an explicit formula for $c_1(n,k)$.
\begin{proposition}
We have
\begin{gather*}c_1(n,n)=1,\\c_1(n,k)=\frac{2^{n-k}k(k+2)\cdots[k+2(n-k-1)]}{(n-k)!},(k<n).
\end{gather*}
\end{proposition}
\begin{proof}
It is well-known that the generating function $g(x)$, for the sequence \begin{equation}\left\{{2n-2\choose n-1}:n=1,2,\ldots\right\},\end{equation} is $g(x)=\frac{1}{\sqrt{1-4x}}$. We know that $c_1(n,k)$ is the coefficient of $x^n$ in the expansion of $[xg(x)]^k$ into powers of $x$.
The formula follows by the use of Taylor expansion for the binomial series.
\end{proof}
The fact that $c_1(n,1)=f_0(n)$ leads to the following identity:
\begin{identity} For each $n\geq 1$, we have
\begin{equation*}\prod_{i=1}^n(n+i)=2^{n}(2n-1)!!.
\end{equation*}
\end{identity}
We also consider the particular case $k=2$. Then $c_1(n,2)=4^{n-2},(n>1)$. Hence, powers of $4$  have the following property.
\begin{corollary} For $n\geq 2$, the number $4^{n-2}$ is the number of ternary words of length $2n-3$ in which one letter is $2$ and in each binary subword, the number of ones and zeros are equal.
\end{corollary}

The number $c_1(n,k)$ may be interpreted in terms of lattice paths.
Namely, It is a well-known fact that $f_0(n)$ is the number of lattice paths from $(0,0)$ to $(n-1,n-1)$ using the steps $(0,1)$ and $(1,0)$. We may consider the symbol $x$ in Proposition \ref{alf} as the $(1,1)$-step possible only on the main diagonal.
We thus obtain the following Euler-type identity.
\begin{identity} The following sets have the same number of elements.
\begin{enumerate}
\item The set of ternary words  having $k-1$ letters equal $2$, in which each binary subword has the same number of zeros and ones.
\item The set of the lattice paths from $(0,0)$ to $(n+k-2,n+k-2)$ using steps $(0,1),(1,0)$ and $k-1$ steps $(1,1)$ possible only on  the main diagonal.
\end{enumerate}
\end{identity}

Our final example is the case $f_0(n)={2n-1\choose n},(n=1,2,\ldots)$. Combinatorially, $f_0(n)$ is the number of binary words of length $2n-1$ in which the number of ones is by $1$ greater than the number of zeros. We  let $\mathcal P_4$ denote this property.
Obviously, the empty word does not satisfy this condition, so that, Proposition \ref{alf} can not be applied.

 We  count the number of words of the form
$w_{i_1},2,w_{i_2},2,\ldots w_{i_{k-1}},2,w_{i_k}$ when $i_1+i_2+\cdots+i_k=n$, and $l(w_i)=2i-1$, for $i\in\{i_1,i_2,\ldots,i_k\}$.
It is easy to see that the following result holds:
\begin{corollary} The number $c_1(n,k)$ is number of ternary words of length $n+k-1$ having the following properties:
\begin{enumerate}
\item No word either begins or ends with $2$.
\item No two $2$ can be adjacent.
\item Each binary subword satisfies $\mathcal P_4$.
 \end{enumerate}
\end{corollary}
It is known that a generating function for the sequence $f_0(1),f_0(2),\ldots$ is
\begin{equation*}g(x)=\frac{1}{2x\sqrt{1-4x}}-\frac{1}{2x}.\end{equation*}
To obtain an explicit formula for $c_1(n,k)$, we have to expand $[xg(x)]^k$
into powers of $x$.
Using the binomial formula and the expansion  of the binomial series, we obtain
\begin{equation}\label{hehe}
[xg(x)]^k=\frac{1}{2^k}\sum_{j=0}^\infty\left(\sum_{i=0}^k(-1)^{k-i}{k\choose i}i(i+2)\cdots(i+2j-2)\right)\frac{2^j}{j!}x^{j}.
\end{equation}
Since the least power of $x$ on the left-hand side is $k$, we obtain the following:
\begin{identity}
\begin{equation*}
\sum_{i=0}^k(-1)^{k-i}{k\choose i}i(i+2)\cdots(i+2j-2),(j<k).
\end{equation*}
\end{identity}
\begin{remark} The identity may easily  be proved directly.
\end{remark}
We thus obtain
\begin{proposition} The following formula holds:
\begin{equation*}
c_1(n,k)=\frac{2^{n-k}}{n!}\sum_{i=0}^k(-1)^{k-i}{k\choose i}\cdot\prod_{t=0}^{n-1}(i+2t).
\end{equation*}
\end{proposition}
From equation (\ref{hehe}), we obtain
\begin{identity} The following formula holds:
\begin{equation*}\prod_{i=1}^{n}(n+i-1)=2^{n-1}(2n-1)!!.\end{equation*}
\end{identity}
\begin{remark} We note that a number of other combinatorial interpretations of our results may be found in Sloane's~\cite{slo}.
\end{remark}

\end{document}